\documentclass[a4paper,12pt]{article}
\usepackage[latin1]{inputenc}
\usepackage{amsfonts}
\usepackage{amsthm}
\usepackage{amsmath}
\usepackage{amsfonts}
\usepackage{latexsym}
\usepackage{amssymb}
\usepackage{dsfont}
\usepackage{graphicx}
\usepackage[usenames]{color}

\newtheorem{fed}{Definition}[section]
\newtheorem*{fed*}{Definition}
\newtheorem*{feds*}{Definitions}
\newtheorem{teo}[fed]{Theorem}
\newtheorem*{teo*}{Theorem}
\newtheorem{lem}[fed]{Lemma}
\newtheorem{cor}[fed]{Corollary}
\newtheorem{pro}[fed]{Proposition}
\theoremstyle{definition}
\newtheorem{rem}[fed]{Remark}

\newtheorem*{rems*}{Remarks}

\def\coma{\, , \, }

\def\py{\peso{and}}
\newcommand{\peso}[1]{ \quad \text{ #1 } \quad }
\def\Par{\big(\,}
\def\Pal{\big)\,}

\def\n0{n_{ \text{\rm \tiny o}}}

\def\suml{\sum\limits}

\def\QEDP{\tag*{\QED}}

\def\bce{\begin{center}}
\def\ece{\end{center}}

\def\py{\peso{and}}
\def\rk{\text{\rm rk}}
\def\noi{\noindent}
\def\cF{\mathcal F}
\def\cG{\mathcal G}
\def\QED{\hfill $\square$}
\def\EOE{\hfill $\triangle$}

\def\uno{\mathds{1}}
\def\bm{\left[\begin{array}}
\def\em{\end{array}\right]}
\def\ben{\begin{enumerate}}
\def\een{\end{enumerate}}
\def\bit{\begin{itemize}}
\def\eit{\end{itemize}}
\def\barr{\begin{array}}
\def\earr{\end{array}}
\def\igdef{\ \stackrel{\mbox{\tiny{def}}}{=}\ }

\def\precski{\ \stackrel{\mbox{\tiny{Lidskii}}}{\prec}\ }

\def\la{\lambda}

\def\N{\mathbb{N}}
\def\R{\mathbb{R}}
\def\C{\mathbb{C}}
\def\I{\mathbb{I}}

\def\cC{\mathcal{C}}

\def\G{\mathcal{G}}
\def\cH{\mathcal{H}}

\def\cP{\mathcal{P}}

\def\cS{{\cal S}}
\def\cT{{\cal T}}
\def\cM{{\cal M}}

\def\cV{{\cal V}}
\def\cU{{\cal U}}

\def\ese{\mathcal{S}}
\def\ete{\mathcal{T}}

\def\orto{^\perp}
\def\inc{\subseteq}

\def\rai{^{1/2}}

\def\api{\langle}
\def\cpi{\rangle}

\def\ua{^\uparrow}
\def\da{^\downarrow}

\def\spr{\text{\rm Spr}}
\def\sprm{\text{\rm Spr}^+}

 \DeclareMathOperator{\tr}{tr}

\DeclareMathOperator{\leqp}{\leqslant}

\def\H{{\cal H}}

\newcommand{\mat}{\mathcal{M}_d(\mathbb{C})}
\newcommand{\matn}{\mathcal{M}_n(\mathbb{C})}

\newcommand{\matsa}{\mathcal{H}(n)}

\newcommand{\matu}{\mathcal{U}(n)}

\newcommand{\matposn}{\matn^+}

\newcommand{\matrec}[1]{\mathcal{M}_{#1} (\mathbb{C})}

\def\beq{\begin{equation}}
\def\eeq{\end{equation}}

\def\pausa{\medskip\noi}

\newcommand{\lambdup}{\lambda^{\uparrow}}

\def\Ax2{\,( S_{E(\cF)^\#_\cV})\hat{}_x }

\newcommand{\nui}[1]{N(#1)}

\begin{document}

\title{The spectral spread of Hermitian matrices}
\author{Pedro Massey, Demetrio Stojanoff and Sebastián Zárate 
\footnote{Partially supported by 
CONICET (PIP 0152 CO), FONCyT (PICT-2015-1505) and UNLP (11X829) 
 e-mail addresses:  massey@mate.unlp.edu.ar , demetrio@mate.unlp.edu.ar , seb4.zarate@gmail.com  }
\\
{\small Centro de Matem\'atica, FCE-UNLP,  La Plata
and IAM-CONICET, Argentina }\\
{\small To the memory of ``el Diego'' Maradona }}
\date{}
\maketitle

\begin{abstract}
\pausa Let $A$ be an $n\times n$ complex Hermitian matrix and let $\la(A)=(\la_1,\ldots,\la_n)\in \R^n$ denote 
the eigenvalues of $A$, counting multiplicities and arranged in non-increasing order. Motivated by  problems arising in the theory of low rank matrix approximation, we study the spectral spread of $A$, denoted $\sprm(A)$, given by $\sprm(A)=(\la_1-\la_{n},\la_2-\la_{n-1},\ldots, 
\la_{k}-\la_{n-k+1})\in \R^k$, where $k=[n/2]$ (integer part). The spectral spread is a vector-valued measure of dispersion of the spectrum of $A$, that allows one to obtain several submajorization inequalities. In the present work we obtain inequalities that are related to Tao's inequality for anti-diagonal blocks of positive semidefinite matrices, Zhan's inequalities for the singular values of differences of positive semidefinite matrices, extremal properties of direct rotations between subspaces, generalized commutators and distances between matrices in the unitary orbit of a Hermitian matrix. 
\end{abstract}



\tableofcontents

\section{Introduction}
In this work we develop several results related to the spectral spread 
of Hermitian matrices (for its definition see Eq. \eqref{spr+ def intro} below). The study of this notion is motivated by some recent problems related to the absolute variation of Ritz values, 
which is one of the many aspects of low rank matrix approximation of
Hermitian matrices (see \cite{Parlett,Saad}). 
Indeed, given an $n\times n$ complex Hermitian matrix $A$ and a $k$-dimensional subspace $\cS$ of $\C^n$, then the Ritz values of $A$ corresponding to $\cS$ are the eigenvalues (counting multiplicities and arranged in non-increasing order)
$\la(S^*AS)=(\la_i(S^*AS))_{i\in\I_k}\in \R^{k}$, where $S$ is an $n\times k$ isometry with range $\cS$ (here $\I_k=\{1,\ldots, k\}$ is an index set). 
If $\cT\subset \C^n$ is another $k$-dimensional subspace then, the absolute variation
of the Ritz values of $A$ related to $\cS$ and $\cT$ is the vector
$$(|\la_i(S^*AS)-\la_i(T^*AT)|)_{i\in\I_k}\in\R^k$$ where $T$ is an $n\times k$ isometry with range $\cT$. 
This topic has been extensively studied (see 
\cite{AKPRitz,BosDr,AKFEM,AKMaj,AKProxy,LiLi,MSZ1,TeLuLi,ZAK}). One of the major problems in this context is to obtain upper bounds for the variation of the Ritz values in terms of some measure of the spread of the spectrum of $A$ and some measure of the distance between the subspaces $\cS$ and $\cT$. As a natural vector-valued measure of distance between the subspaces it is usually considered $\Theta(\cS,\cT)=(\theta_i)_{i\in\I_k}\in\R^k$ the so-called vector of principal angles between $\cS$ and $\cT$ (for details see Section 
\ref{sec prelis}). On the other hand, as a measure of the spread of the spectrum of 
$A$, many authors have considered the diameter i.e., $\la_{\max} (A)-\la_{\min}(A)\geq 0$.

\pausa In \cite{AKFEM}, A. Knyazev et.al. considered the vector valued measure of the spread of the spectrum of $A$ denoted by $\sprm(A)$ given by 
\beq\label{spr+ def intro}
\sprm (A) =(\la_i(A)-\la_{n-i+1}(A))_{i\in \I_{[n/2]}} 
=\big(\la_i(A)-\la\ua_i(A)\, \big)_{i\in \I_{[n/2]}} \in \R_{\ge 0}^{[n/2]} \, ,
\eeq 
where $\la\ua(A)=(\la\ua_i(A))_{i\in\I_n}$ is the vector of eigenvalues, but arranged in non-decreasing order, and 
$[n/2]$ denotes the integer part of $n/2$. Similarly, we can consider
\beq\label{spr def intro}
\spr (A) =\la(A)+\la(-A) 
= \big(\la_i(A)-\la\ua_i(A)\, \big)_{i\in \I_n} \in \R^n \, .
\eeq
Notice that $\spr(A)\in\R^n$ is a symmetric vector, that is
$\spr_i(A)=-\spr_{n-i+1}(A)$, for $i\in\I_n$.
 Moreover, using Weyl's inequality 
it turns out that $\spr(A)$ is a vector-valued measure of the diameter of 
the unitary orbit of $A$. With these notions, in \cite{AKFEM} the authors 
conjectured that 
\beq \label{eq intro conjK}
(|\la_i(S^*AS)-\la_i(T^*AT)|)_{i\in\I_k}\prec_w (\sin(\theta_i)\ \sprm_i(A))_{i\in\I_m}
\eeq where the previous inequality is with respect to submajorization and
$m=\min\{k,[n/2]\}$. It is well known that submajorization relations (as that conjectured 
in Eq. \eqref{eq intro conjK}) imply inequalities with respect to arbitrary 
unitarily invariant norms, and tracial inequalities involving convex non-decreasing functions (see Sections \ref{sec prelis} and \ref{sec appendixity} for details).

\pausa In \cite{MSZ1} we obtained some inequalities related to  the variation of 
Ritz values that are weak versions of Eq. \eqref{eq intro conjK}. It turns out that Eq. \eqref{eq intro conjK} encodes some subtle aspects of the spectral spread
$\sprm(A)$ that are still not understood. Indeed, at that time we realized that although natural, the spectral spread seemed not to have been considered in the literature. 
Thus, on the one hand we consider it is interesting to develop some of its basic features. On the other hand, motivated by the seminal ideas from \cite{AKPRitz,AKFEM,ZAK,ZK} 
in this work we propose some inequalities involving the spectral spread. For example, given the $k$-dimensional 
subspace $\cS\subset \C^n$ and $n\times n$ complex Hermitian matrix $A$ as before, 
if we let $\cT=e^{i\,A}\,\cS$ then numerical experiments supported the submajorization inequality
\beq \label{eq intro a1}
\Theta(\cS\coma \cT)\prec_w \frac 1 2 \, \sprm(A)\,.
\eeq
It turns out that Eq. \eqref{eq intro a1} (that reflects some extremal properties of direct rotations, as introduced by Davis and Kahan in \cite{DavKah}) is equivalent to
the following submajorization inequality: 
if $n = k+r$ and we let 
$A$ be the $n\times n$ complex Hermitian matrix with blocks
\beq\label{eq intro a2}
A =\bm{cc}A_1&B \\B^*&A_2 \em \barr {c}\C^k\\ \C^r\earr \peso{then} 2\,s(B)\prec_w \sprm(A) \ ,
\eeq
where $s(B)=\la(\,(B^*B)^{1/2})\in  \R_{\ge 0}^{r} $ 
denotes the vector of singular values of 
the $k\times r$ matrix $B$, i.e. the eigenvalues of the modulus $|B|=(B^*B)^{1/2}$.

\pausa In this paper we prove Eq. \eqref{eq intro a2} (hence, also Eq. \eqref{eq intro a1}\,), which we consider as a key inequality for the spectral spread; we point out that this inequality is sharp. In turn, Eq. \eqref{eq intro a2} connects our work with Tao's work \cite{Tao}, where he showed that 
for a positive semidefinite matrix $A$ with blocks
\beq \label{eq intro Tao1}
A=\bm{cc}A_1&B \\B^*&A_2 \em \barr {c}\C^k\\ \C^r\earr \peso{it holds that}
2\,s_i(B)\leq s_i(A\oplus A)\peso{for} i\in \I_k\ .
\eeq 
Notice that although 
Eq. \eqref{eq intro a2} provides a spectral relation that is weaker than 
the entry-wise inequalities in Eq. \eqref{eq intro Tao1}, our upper bound 
for positive semidefinite $A$ satisfies
$$
\sprm_i(A)=\la_i(A)-\la_{n-i+1}(A)\leq \la_i(A)\leq 
\la_i(A\oplus A) = s_i(A\oplus A)  \ , \peso{for} i\in\I_{[n/2]} \ . 
$$
For example, in case $A=a\,I$ 
then $B=0$ a fact that 
is reflected by $\sprm(A)=0$, while $s_i(A\oplus A)=a$, for $i\in\I_n$. On the other hand, Eq. \eqref{eq intro a2} is valid in the (general) Hermitian case.
Again, by \cite{AuKitt,Tao} it turns out that our work is connected with Zhan's inequality for the singular values of the difference of positive semidefinite matrices in \cite{Zhan00,Zhan02,Zhan04}.
Motivated by these facts, we show that Eq. \eqref{eq intro a2} is equivalent to the inequality
\beq \label{eq intro a3}
 s(A_1-A_2)\prec_w \sprm(A_1\oplus A_2)
\eeq  for arbitrary complex Hermitian matrices $A_1$ and $A_2$ 
of the same size. In this case, Eq. \eqref{eq intro a3} and Zhan's inequality can be compared in a way similar to the comparison between Eq. \eqref{eq intro a2} and Eq. \eqref{eq intro Tao1}.
We point out that the upper bounds obtained in Eq. \eqref{eq intro a2}
and Eq. \eqref{eq intro a3} are vectors that are invariant 
by translations $M\mapsto M+\la\,I$ for the matrices $M$ involved, 
as opposed to previous upper bounds that are based on singular values, 
i.e. Eq, \eqref{eq intro Tao1} and Zhan's inequality. Since the vectors that are being bounded in these theorems are also invariant under the corresponding translations, we consider that the upper bounds in terms of the spread are particularly well suited in this context.

\pausa
On the other hand, it turns out that Eq. \eqref{eq intro a3} can be extended to the context 
of generalized commutators as follows: given $n\times n$ complex 
matrices $A_1,\,A_2,\,X$ such that $A_1$ and $A_2$ are Hermitian then
\beq \label{eq intro a4}
s(A_1\,X-X\,A_2)\prec_w \big(\, s_i(X)\ \sprm_i(A_1\oplus A_2)\,\big)_{i\in\I_n}\,.
\eeq
This last inequality connects our work with a series of papers dealing 
with (even more general) inequalities for singular values 
of generalized commutators 
\cite{Hirz09,HirKitt10,Kitt,Kitt2008,Kitt09,KittWangDu}. 
We point out that in the previous works, the authors obtain entry-wise 
upper bounds for the singular values of generalized commutators in terms of 
singular values and some measures of the spread of related matrices (among other type of inequalities).
Our results are obtained in terms of weaker submajorization relations, but the upper
bound in Eq. \eqref{eq intro a4} involves the complete list of singular values of $X$ and the full spectral spread of $A_1\oplus A_2$. 
On the other hand, we point out
that Eq. \eqref{eq intro a4} holds for arbitrary Hermitian matrices $A_1$ and $A_2$.

\pausa We point out that Eq. \eqref{eq intro a2} together with some of its equivalent forms allow one to develop inequalities related to  Eq. \eqref{eq intro conjK} (see \cite{MSZ2}). 
In the last section of the paper we show the equivalence of the inequalities in Eqs. \eqref{eq intro a1}, \eqref{eq intro a2}, \eqref{eq intro a3} and \eqref{eq intro a4}.

\section{Spectral spread}

Although natural, the spectral spread of Hermitian matrices 
seems not to have been considered in the literature, after being introduced in \cite{AKFEM}. Thus, we begin with a preliminary section (with notations and basic definitions), and then we present some basic results related to  this notion. After this, we consider a submajorization inequality for the spectral spread that plays a key role in our work. We obtain some consequences of this inequality related to  Zhan's \cite{Zhan00} inequality and Davis-Kahan's notion of direct rotation between subspaces \cite{DavKah} (see also \cite{QZLi}).

\subsection{Preliminaries}\label{sec prelis}

In this section we give the basic notation and definitions that we use throughout our work.  
In the Appendix (Section \ref{sec appendixity}) we state several well known results of Matrix Analysis involving the notions described below. 

\pausa
{\bf Notation and terminology}. We let $\mathcal{M}_{n,k}(\C)$ be the space of complex $n\times k$ matrices 
and write $\mathcal{M}_{n,n}(\C)=\matn$ for the algebra of $n\times n$ complex matrices. 
We denote by $\H(n)\subset \matn$ the real subspace of Hermitian matrices and by $\matn^+$, the cone of
positive semi-definite matrices. Also, 
$\cG l(n)\subset \matn$ and $\mathcal{U}(n)$  denote the groups of invertible and unitary matrices respectively, 
and $\G l (n)^+ =\G l(n)\cap \matn^+$.
A norm $N$ in $\matn$ is unitarily invariant (briefly u.i.n.) if 
$\nui{UAV}=\nui{A}$, for every $A\in\matn$ and $U,\, V\in\mathcal{U}(n)$.

\pausa
For $n\in\N$, let $\I_n=\{1,\ldots,n\}$. 
Given a vector $x\in\C^n$ we denote by $D_x$ the diagonal matrix in $\matn$ whose main diagonal is $x$.
Given $x=(x_i)_{i\in\I_n}\in\R^n$ we denote by $x\da=(x_i\da)_{i\in\I_n}$ the vector obtained by 
rearranging the entries of $x$ in non-increasing order. We also use the notation
$(\R^n)\da=\{x\in\R^n\ :\ x=x\da \}$ and $(\R_{\geq 0}^n)\da=
\{x\in\R_{\geq 0}^n\ :\ x=x\da \}$. Similarly we define $x\ua$ and $(\R^n)\ua$. 
For $r\in\N$, we let $\uno_r=(1,\ldots,1)\in\R^r$.

\pausa
 Given a matrix $A\in\cH(n)$ we denote by $\la(A)=(\la_i(A))_{i\in\I_n}\in (\R^n)\da$ 
the eigenvalues of $A$ counting multiplicities and arranged in 
non-increasing order. Similarly, we denote by 
 $\la\ua(A)=(\la\ua_i(A))_{i\in\I_n}\in (\R^n)\ua$.   
For $B\in\cM_{k\coma r}(\C)$ we let $s(B)=\la(|B|)\in (\R_{\geq 0}^r)\da$ denote the singular values of $B$, i.e. the eigenvalues of $|B|=(B^*B)^{1/2}\in\cM_r(\C)^+$. 

\pausa Arithmetic operations with vectors are performed entry-wise in the following sense:
 in case $x=(x_i)_{i\in\I_k}\in \C^k,\,y=(y_i)_{i\in\I_r}\in \C^r $ 
then $x+y=(x_i+y_i)_{i\in\I_m}$, $x\, y=(x_i\,y_i)_{i\in\I_m}$ and (assuming that $y_i\neq 0$, for $i\in\I_r$) $x/y=(x_i/y_i)_{i\in\I_m}$,  
where $m=\min\{k,\,r\}$. Moreover, if we assume further that $x,\,y\in\R^k$ then we write $x\leqp y$  ($\leqp$, different from the notation $\le$) whenever 
$x_i\leq y_i$, for $i\in\I_k$.

\pausa Given $f: I \rightarrow \R$, where $I\subseteq \R$ is an interval, and $z=(z_i)_{i\in\I_k}\in I^k$ we denote $f(z)=(f(z_i))_{i\in\I_k}\in\R^k$. For example, $|z|=(|z_i|)_{i\in\I_k}$, $\sin(z)=(\sin(z_i))_{i\in\I_k}$.
\EOE

\pausa Next we recall the notion of majorization between vectors, that will play a central role throughout our work.
\begin{fed}\rm 
Let $x,\, y\in\R^k$. We say that $x$ is
{\it submajorized} by $y$, and write $x\prec_w y$,  if
$$
\suml_{i=1}^j x^\downarrow _i\leq \suml_{i=1}^j y^\downarrow _i \peso{for} j\in\I_k\,. 
 $$ If $x\prec_w y$ and $\tr x \igdef \suml_{i=1}^kx_i=
 \tr y$,  then we say that $x$ is
{\it majorized} by $y$, and write $x\prec y$. 
\end{fed}

\pausa We point out that (sub)majorization is a preorder relation in $\R^k$ that plays a central role in matrix analysis (see Section \ref{sec appendixity}).

\begin{rem}\label{rem acuerdos}
Let $x\in\R_{\geq 0}^k$ and $y\in \R_{\geq 0}^h$ be two vector with {\it non-negative} entries (of different sizes). We extend the notion of submajorization between $x$ and $y$ in the following sense: 
\beq\label{size}
x\prec_w y \peso{if} \begin{cases} (x\coma 0_{h-k}) \prec_w \quad \ y  & \peso{for} k<h \\ 
\quad \quad x  \quad \  \prec_w (y\coma 0_{k-h})   & \peso{for} h<k \end{cases} 
\quad , 
\eeq 
where $0_n$ denotes the zero vector of $\R^n$. \EOE
\end{rem}

\subsection{Basic properties of the spectral spread}
In this section we present several basic properties of the spectral spread, and describe the relationship between this notion and singular values (and other usual notions of matrix analysis). 

\begin{fed}\label{defi spr}\rm 
Let $A\in \cH(n)$. Consider the {\it full spectral spread} of $A$, given by 
\beq\label{spr def}
\spr (A) \igdef \la(A)+\la(-A) = \big(\la\da_i(A)-\la\ua_i(A)\, \big)_{i\in \I_n} \in (\R^n)\da \ .
\eeq
Denote by $k = [\frac n2]$ (integer part). 
We also consider the {\it spectral spread of} $A$, that is the non-negative part of $\spr(A)$:
\beq\label{spr+ def}
\sprm (A) =  \big(\spr_i(A)\, \big)_{i\in \I_k}
=\big(\la_i(A)-\la\ua_i(A)\, \big)_{i\in \I_k} \in (\R_{\ge 0}^k)\da \ .
\eeq
\end{fed}

\begin {rem}\label{rem cosas elemen sobre spread} \rm
Let $n = 2k $ or $2k+1$, and $A 
\in \cH(n)$. Let $\la(A)=(\la_i)_{i\in\I_n}$ and $\la\ua(A)=(\la\ua_i)_{i\in \I_n}$.
\ben
\item For every $t\in \R$ we have that
\beq\label{con -t}
\spr(A-tI)=\spr(A) \py \sprm(A-tI)=\sprm(A) \ . 
\eeq
\item 
For $r\in \I_k$, it is well known (see \cite{bhatia}) that 
$$
\barr{rl}
\suml _{i \in \I_r} \la_i(A) & = 
\max \, \left\{\,\suml_{i \in \I_r}  \api A\,x_i\coma x_i\cpi : 
\{x_i\}_{i\in\I_r} \text{ is an ONS }\, \right\} \py \\&\\
-\suml _{i \in \I_r} \la_i\ua(A) & = \suml _{i \in \I_r} \la_i(-A)
=\max \, \left\{\,-\suml_{i \in \I_r}  \api A\,y_i\coma y_i\cpi : 
\{y_i\}_{i\in\I_r} \text{ is an ONS }\, \right\}
\earr
$$
Then, for each $r\in \I_k$ we have that 
\beq\label{spr max}
\sum_{i \in \I_r} \spr_i(A) =
\max \, \left\{\,\suml_{i \in \I_{r}}  \api A\,x_i\coma x_i\cpi 
- \api A\,y_i\coma y_i\cpi  : 
\{x_i\} \text{ and $\{y_i\}$  are ONS's }\, \right\} 
\eeq
\item
It is straightforward to check that $\max\{|\la_i(A)|,|\la\ua _{i}(A)|\}\leq s_i(A)$, for $i\in \I_k$; hence, we conclude that 
\beq\label{spr vs s}
\sprm_i(A)\leq   |\la_i|+|\la\ua _{i}| \le 2s_i(A) \peso{for every} i \in \I_k\ .
\eeq
In the positive case, we have that:
\beq\label{spr vs s pos}
A\in \matn^+\implies 
 \spr_i(A) \le \la_i= s_i(A)\peso{for every} i \in \I_k\ .
\eeq
\item 
Conversely, notice that
$$
\sprm_i(A)=\la_i-\la_{n-i+1}=\underbrace{\la_i-\la_{k+1}}_{\geq 0} 
+\underbrace{\la_{k+1}-\la_{n-i+1}}_{\geq 0}\peso{for}i\in\I_k\,.
$$Hence, it follows that 
\beq\label{s vs spr}
s(A-\la_{k+1}\,I) =
\big(\, (\la_i-\la_{k+1}\,)_{i\in \I_{k}}\coma 
(\la_{k+1}-\la_{k+i}\,)_{i\in \I_{n-k}} \,\big)\da 
 \prec \sprm(A) \in \R_{\ge 0}^k 
\ .
\eeq
\item 
On the other hand,  if $B\in \cH(n)$ then 
\beq\label{mayo spr}
\la(A)\prec\la(B) \in\R^{n} \implies \quad \spr(A) \prec \spr(B)
\implies \sprm(A) \prec_w \sprm(B) \ ,
\eeq
which are direct consequences of 
Definition \ref{defi spr} together with Lemma \ref {lem submaj props1}.
\item Consider an isometry $Z\in\cM_{n,r}(\C)$ i.e., such that $Z^*Z=I_r\,$, 
for some $r\in\I_n\,$. 
Then, it is easy to see that 
$Z^*AZ\in\cM_r(\C)$ is a principal submatrix of a unitary conjugate of the matrix $A$.
Therefore, we can apply the interlacing inequalities (see \cite{bhatia}) and get that 
$\la\ua_i(A)\leq \la\ua _i(Z^*AZ)$ and $\la_i(Z^*AZ)\leq  \la_i(A)$ for $i\in\I_r\,$. 
As a consequence,  
\beq \label {comprimido}
\sprm_i(Z^*AZ)\leq \sprm_i(A)\peso{for} i\in \I_{[r/2]} 
\stackrel{\eqref{size}}{\implies}
 \sprm(Z^*AZ)\prec_w\sprm(A)\ .
\eeq

\item If $U\in\cU(n)$ is a unitary matrix then, using Weyl's inequality 
(item 1 in Theorem \ref{teo ah}),  
$$\la(A-U^*AU)\prec\la(A)+\la(-U^*AU)=\la(A)-\la\ua(A)= \spr(A)\,.$$ Moreover, 
if $N$ is a u.i.n. then
$$
\max\,\big\{\,N(A-U^*AU)\ : \ U\in \cU(n)\,\big\}=N(D_{\spr(A)})\,.
$$
That is, $|\spr(A)|\da=\sprm(A\oplus A)=(\sprm A \coma \sprm A)\da$ 
(with an extra $0$ if $n$ is odd) allows us 
to compute the diameter of the unitary orbit of $A$ (with respect to any unitarily 
invariant norm). In this sense, $\spr(A)$ can be considered 
as a vector valued measure of the diameter of the unitary orbit of $A$.
\EOE
\een
\end{rem}

\begin{pro}\label{pro com spr y s} 
Let $A \in \cH(n)$. Then \rm 
\beq\label{vale esto} 
\frac 1 2 \,|\spr(A)|\da=\frac 1 2 \,\sprm(A\oplus A)\prec_w  s(A)\ .
\eeq
\end{pro}
\proof Denote by  $\la (A) = (\la_i)_{i\in \I_{n}} \in (\R^{n})\da\,$. Then 
for $1\le i\le k =[\frac n2]$, 
$$
\sprm_{2i-1}(A\oplus A) = \sprm_{2i}(A\oplus A) = \sprm_{i}(A) = 
\la_i - \la_{n-i+1} 
\le |\la_i |+| \la_{n-i+1}|\,.
$$
On the one hand we have that $\sprm(A\oplus A)$ equals
$$
\sprm(A\oplus A)=(\sprm A \coma \sprm A)\da\in \R^n\peso{if}  n=2k $$
 $$\peso{or}  \sprm(A\oplus A)=(\sprm A\coma \sprm A \coma 0)\da \in \R^n \peso{if}  n=2k+1\,. 
$$
Therefore, in order to check that $\sprm(A\oplus A)\prec_w 2\,s(A)$ it suffices to check that
$$
\sum_{i\in\I_{2r}}\sprm(A\oplus A)\leq 2\ \sum_{i\in\I_{2r}} s(A)\peso{for} r\in\I_{k}\,,
$$ for the even cases $2r$. Hence, if $r\in\I_k$, we have that 
$$
\sum_{i\in\I_{2r}} \sprm(A\oplus A)\leq 2 \ \sum_{i\in\I_{r}} |\la_i|+  |\la_{n-i+1}|\leq 
2\ \sum_{i\in\I_{2r}} s_i(A)$$
since $s(A)= (|\la_i|)_{i\in \I_{n}} \da$, 
and $\suml_{i \in \I_{2r}} s_i(A) = \max\{\suml _{j\in \mathbb{F}} 
|\la_j| : |\mathbb{F} | =2r\}$. Notice that Eq. \eqref{vale esto} follows from this fact. 
\QED

\pausa
Notice that if $A\in\cH(n)$ then, by items 1. and 4. in Remark \ref{rem cosas elemen sobre spread} and Proposition \ref{pro com spr y s}, 
\beq\label{vale esto 2} 
\frac{1}{2}\, \sprm(A\oplus A)\prec s\big(\,A - \la_{k}(A)\,I\,\big)
\stackrel{\eqref{s vs spr}}{\prec} \sprm(A) \ .
\eeq
Our next result is a spectral spread version of Lidskii's inequality.

\begin{pro}\label{spreadskii aditivo}
Let $A,B\in \H(n)$. Then 
\rm
\beq
\spr(A)-\spr(B) \prec \spr(A-B)\prec \spr(A)-\spr\ua(B) =\spr(A)+\spr(B) \ .
\eeq 
\end{pro}
\proof
By Lidskii's additive inequality and item 3 of Lemma \ref {lem submaj props1} (see the Appendix),  
$$ 
\barr{rl} \spr(A)-\spr(B) & \ \ = \ 
\lambda(A)-\lambda(B)+\lambda(-A)-\lambda(-B) \\&\\
 & \precski \lambda(A-B)+\lambda(-(A-B))
=\spr(A-B) \ .
\earr
$$
For the other inequality,  note that $\spr\ua(B)=\lambdup(B)-\lambda(B)=-\spr(B)$. Therefore 
$$
\barr{rl}\spr(A-B) 
& \ \ = \ \lambda(A-B)+\lambda(B-A) \\&\\
& \precski \lambda(A)-\la\ua(B)+\lambda(B)-\lambdup(A)
 =\spr(A)-\spr\ua(B)\ , 
\earr  
$$
where we have used again item 3 of Lemma \ref {lem submaj props1}. 
\QED

\subsection{A key inequality for the spectral spread}\label{sec 31}

The following inequality plays a central role in our present work.

\begin{teo}\label{vale con 2 gral}
Let $A =\bm{cc}A_1&B \\B^*&A_2 \em \barr {c}\C^k\\ \C^r\earr \in \cH(k+r)$. Then 
\rm
\beq\label{con 2 gral}
2\,s(B)\prec_w \sprm(A) \ . 
\eeq
\end{teo}
\proof
Consider 
$U =\bm{cc}I&0 \\0&-I \em \barr {c}\C^k\\ \C^r\earr \in \cU(k+r)$. Then
$$
s(UA-AU)=s(A-U^*AU)=|\la(A-U^*AU)|\da\,.
$$ 
By Lidskii's inequality (Theorem \ref{teo ah}) we have that 
$$
\la(A-U^*A\,U)\prec \la(A)-\la\ua(U^*AU)=\la(A)-\la\ua(A) \,. 
$$ 
Using Remark \ref{convfunction} we get that 
$$
s(UA-AU)=|\la(A-U^*A\,U)|\da\prec_w|\la(A)-\la\ua(A)|\da
=\Par \sprm(A)\coma \sprm(A)\,\Pal\da 
$$ 
(or $\Par \sprm(A)\coma \sprm(A)\coma 0\Pal\da$ if $k+r$ is odd).  Using Proposition \ref{hat trick como en el futbol} and noticing that 
$$
UA-AU=\bm{cc}0&2\,B \\-2\,B^*&0 \em \barr {c}\C^k\\ \C^r\earr 
\implies s(UA-AU)=2\,\Par s(B)\coma s(B) \coma 0_{|k-r|}\,\Pal\da\ ,
$$ 
where $s(B)$ has size $\min\,\{r\coma k\} \le [\frac{k+r}{2}]\,$ (we use that 
$\rk \, B \le \min\,\{r\coma k\}$). 
We conclude that 
\beq 
2\,\Par s(B)\coma s(B) \coma 0_{|k-r|}\,\Pal\da
\prec_w \Par\sprm(A)\coma \sprm(A)\,\Pal\da\implies 2\,s(B)\prec_w\sprm(A)\ .
\QEDP \eeq

\medskip

\begin{rem}\label{es charp}
We point out that the inequality in Eq. \eqref{con 2 gral} is sharp. Indeed, consider ($k=r$)
$$
\peso{if} A =\bm{cc} 0&B \\B^*&0 \em \barr {c}\C^k\\ \C^k\earr 
\implies \la(A)=(s(B)\coma -s(B))\da \peso{(see Proposition \ref{hat trick como en el futbol})}
\,.$$
Thus, in this case we have the equality $2\,s(B) = \sprm(A)$.
\EOE
\end{rem}

\begin{rem}\label{con Tao}
Let $A 
 \in \cH(k+r)$ 
be as in Theorem \ref{vale con 2 gral}.  In \cite{Tao}, 
Y. Tao proved that 
if $A \in \matrec{k+r}^+$, then 
\beq\label{Tao}
2s_j(B)\leq \la_j(A) = s_j(A) \peso{for every} j \in \I_{k+r} \  .
\eeq
In the positive case we have that $\sprm (A) \leqp\la(A)$. 
Nevertheless, 
the inequality 
\beq\label{no vale}
2s_j(B)\leq \sprm_j(A) \peso{for every} j \in \I_{[\frac{k+r}2]} 
\eeq
is not true, even for positive semidefinite matrices $A$. 
For example take 
$$A=\begin{pmatrix}
2&1&0&1\\
1&2&1&0\\
0&1&3&1\\
1&0&1&3
\end{pmatrix} \in \matrec{2+2}^+$$
Then $s(B)=(1,1)$, and $\lambda(A)=(4.61\coma 2.61 \coma 2.38 \coma 0.39)$. Hence, Eq. \eqref{no vale} fails for $j=2$. In particular, Theorem \ref{vale con 2 gral} does not imply Tao's inequality. Neither Tao's inequality implies Theorem \ref{vale con 2 gral}, 
as the Eq. \eqref{vale esto 2} could suggests, because
$A-\la_{[\frac{k+r}2]}(A)\,I \notin \matrec{k+r}^+$.

\pausa On the other hand, when $A\in \cM_{k+r}(\C)^+$ then both Tao's and our result are applicable. In this case, if $N$ is a unitarily invariant norm 
and $\mu = (\sprm(A),\,0)\in \R^{k+r}$,  then 
\beq\label{eq consec de teo y tao}
2\, N\begin{pmatrix} 0& B\\ 0&0\end{pmatrix}
\leq N(D_\mu) 
\leq N(A)\ ,
\eeq 
where 
$D_\mu $ is the diagonal matrix with main diagonal $\mu$. 
Indeed, since $\mu 
\leqp \la(A)=s(A)$, then 
$$
2\,(s(B),\,0)\stackrel{\eqref{vale con 2 gral}}{\prec_w}(\sprm (A),\,0) \prec_w s(A)
$$ 
and Eq. \eqref{eq consec de teo y tao} follows from these relations. We also point out that 
in (the generic) case 
$A\in \G l (k+r)^+$
we get a strict inequality 
$N(D_\mu) < N(A)$, for an arbitrary strictly convex u.i.n. $N$. On the other hand,  
notice that Theorem \ref{vale con 2 gral} also applies in case $A$ is an arbitrary (not necessarily positive semidefinite) Hermitian matrix.
\EOE
\end{rem}

\medskip

\begin {cor}\label{antidiagonal}
Let 
$A= \bm{cc}A_1&B \\B^*&A_2 \em \barr {c}\C^k\\ \C^r\earr \in \cH(k+r)$. Then 
\beq\label{sin diag}
\sprm \bm{cc} 0&B \\B^*&0 \em \prec_w \sprm (A)
\py 
\sprm \bm{cc} A_1 &0 \\0&A_2 \em \prec_w \sprm (A)
\ . 
\eeq
\end{cor}
\proof
By Proposition \ref{hat trick como en el futbol} we get that 
$\sprm \bm{cc} 0&B \\B^*&0 \em = 2\, s(B)$, 
(modulo some zeros at the end to equate sizes). 
Applying Eq. \eqref{con 2 gral}, we get the first 
inequality in Eq. \eqref{sin diag}. The second follows 
from Eq.'s \eqref{mayo spr} and \eqref{pinch}. 
\QED

\medskip

\pausa
In what follows we develop two different applications of the key inequality 
in Theorem \ref{vale con 2 gral}, that are related to  a companion to Zhan's inequality from \cite{Zhan00} and with results related to  direct rotations between subspaces from \cite{DavKah,QZLi} in the finite dimensional context. We begin with the following

\medskip

\begin{rem} In \cite{Zhan00} Zhan showed that given $A_1,\,A_2\in \cM_n(\C)^+$ then we have that 
\beq\label{eq zhan1}
s_i(A_1-A_2)\leq s_i(A_1\oplus A_2)\peso{for every} i \in \I_n \ .
\eeq These fundamental inequalities are known to be equivalent to other central results in matrix analysis 
(see \cite{AuKitt,Tao,Zhan02,Zhan04}). On the other hand, these entry-wise inequalities between singular values allow one to get operator inequalities the form $A_1-A_2\leq V^*(A_1\oplus A_2)V$, for suitable contractions $V\in \cM_{2n,n}(\C)$. Moreover, Eq. \eqref{eq zhan1} imply the (weaker) submajorization relation $s(A_1-A_2)\prec_w s(A_1\oplus A_2)$, as in Eq. \eqref {size}.

\pausa In case $C,\,D\in\mat$ are arbitrary, then the previous inequality does not hold. Nevertheless, a small modification of arguments in \cite{Zhan00} imply the inequalities
\beq\label{Evita}
s_i(C-D)  \ \le \ 2\ s_i(C\oplus D) \peso{for every} i \in \I_n \ .
\eeq The previous inequalities are sharp, even for Hermitian matrices $C,\,D\in\cH(n)$ (take $C\in\cM_n(\C)^+$ and let $D=-C$). As before, these entry-wise inequalities between singular values allow us to get some related operator inequalities and submajorization relations.

\pausa On the other hand, if we are interested in norm inequalities with respect to unitarily invariant norms then we can improve the upper bounds derived from Eqs. \eqref{eq zhan1} and \eqref{Evita} for arbitrary Hermitian matrices as follows.
\EOE
\end{rem}

\medskip

\begin{teo}\label{lem prelim1}
 Let $A_1,\, A_2\in\cH(n)$ be Hermitian matrices. Then \rm
\beq\label{Peron}
s(A_1-A_2) \prec_w \sprm \,(A_1\oplus A_2)
\ .
\eeq
\end{teo}
\proof
We assume that $A_1\coma A_2 \in \cH(n)$. Let  
$$
Z = \frac1{\sqrt{2}} \ \bm{cc}I&I\\-I&I\em  
\in \cU(2n) \py 
T=\frac 1 2 \  \bm{cc}A_1+A_2 &A_1-A_2\\ A_1-A_2&A_1+A_2 \em \in \cH(2n) \ .
$$
Then  
$ZTZ^* = A_1\oplus A_2$; hence, by Theorem \ref{vale con 2 gral} we see that  $s(A_1-A_2) \prec_w \sprm(T) = \sprm(A_1\oplus A_2 )$.
\qed

\medskip

\begin{rem}\label{rem 213} \rm
 Let $A_1,\, A_2\in\cM_n(\C)^+$. As we have already mentioned, in this case we have that 
$\sprm_i(A_1\oplus A_2)\leq s_i(A_1\oplus A_2)$, for $i\in\I_n\,$. Using this fact and Theorem \ref{lem prelim1} we get that 
$$
N(A_1-A_2)\leq N(D_{\sprm (A_1\oplus A_2)})\leq N(D_{(s_i(A_1\oplus A_2))_{i\in\I_n}})\ ,
$$
for every u.i.n. N. On the other hand, we point out that the entry-wise inequalities 
$$
s_i(A_1-A_2) \le \sprm_i \,(A_1\oplus A_2) \peso{for} i \in \I_n \ ,
$$ are false, even in the positive semidefinite case.
Indeed, take $A_1 = \bm{cc} 3&2\\2&3\em $ and $A_2 = 3 \,I$. Then 
$s(A_1-A_2) = (2\coma 2)$ but 
$$
\la\,(A_1\oplus A_2) =  (5\coma 3\coma 3\coma 1) \implies 
\sprm \,(A_1\oplus A_2) = 
 (4\coma 0) \, . 
$$
We point out that the upper bounds obtained in Theorems \ref{vale con 2 gral} and \ref{lem prelim1}
are vectors that are invariant by translations $M\mapsto M+\la\,I$ for the matrices $M$ involved, 
as opposed to previous upper bounds that are based on singular values, i.e. Eqs. \eqref{Tao} and \eqref {eq zhan1}. Since the vectors that are being bounded in these theorems are also invariant under the corresponding translations, we consider that the upper bounds in terms of the spread are particularly well suited in this context. 

\pausa Finally, we point out that the inequality in Eq. \eqref{Peron} is sharp; indeed, take an arbitrary $A_1\in\matposn$ and let $A_2=-A_1\in \cH(n)$. Then $s(A_1-A_2)=2\,\la(A_1)=\sprm(A_1\oplus A_2)$.
\EOE
\end{rem}

\medskip

\begin{rem}
Let $\cS,\,\cT\subset \C^n$  be $k$-dimensional subspaces and let 
$S,\,T\in \cM_{n,k}(\C)$ and $S_\perp \in \cM_{n\coma n-k}(\C)$ be isometries with ranges 
$R(S)=\cS$, $R(T)=\cT$and  $R(S_\perp)=\cS\orto$.  
The principal angles $\Theta(\cS\coma \cT)=(\theta_j)_{j\in\I_k}\in ([0,\pi/2]^k)\da$ between $\cS$ and $\cT$ are defined by 
\beq\label{angles}
\cos(\theta_j)=s_{k-j+1}(S^*T) \peso{or also by} 
\sin (\theta_j)=s_{j}(T^*S_\perp)  \peso{for} j\in\I_k\  ,
\eeq
(see, for example, \cite{QZLi} for details). 
The principal angles $\Theta(\cS,\cT)$ completely describe the relative position
of these subspaces. On the other hand, they provide a natural 
notion of distance in the Grassmann manifold of all $k$-dimensional subspaces in $\C^n$.

\pausa In \cite{DavKah} Davis and Kahan introduce the fundamental notion of direct 
rotation from $\cS$ onto $\cT$. Briefly, a unitary $U\in \cU(n)$ is a direct rotation
from $\cS$ to $\cT$ if $U\cS=\cT$ and there exist $C_0\in M_k(\C)^+$, $C_1\in M_{n-k}(\C)^+$ and $S_{0}\in \cM_{n-k,k}(\C)$ such that 
$$
W^*U\,W=\bm {cc} C_0& -S_0^*\\  S_0 & C_1 \em\in \cU(n)
$$ for some unitary matrix $W\in\cU(n)$, whose first $k$ columns form an orthonormal basis of $\cS$.
In this case we can write $U=e^{i\,Z}$, where $Z\in\cH(n)$ is such that 
$$
\la(Z)=(\Theta(\ese\coma \ete)^*,-\,\Theta(\ese\coma \ete)^*,0)\da\in (\R^n)\da 
\ ,
$$ 
where $\Theta(\ese\coma \ete)^*=(\theta_i)_{i\in \I_r}\in ((0,\pi/2]^r)\da $
denotes the vector with the positive principal angles, and hence $r\leq [n/2]$.
Direct rotations enjoy some extremal properties that play a central role in the study of metric properties in the Grassmann manifold of all $k$-dimensional subspaces in $\C^n$. 
Indeed, Davis and Kahan showed in \cite{DavKah} that if $V\in\cU(n)$ is such that 
$V\,\cS=\cT$ and $U$ is a direct rotation from $\cS$ onto $\cT$ then
\beq\label{eq desi ver de DK}
2\,\sin(\theta_j/2)=s_j((1-U)|_\cS)\leq s_j((1-V)|_\cS) \peso{for} j\in\I_k\,.
\eeq
On the other hand, if $X\in\cH(n)$ is such that $\la_j(X)\in [-\pi,\pi]$, for $j\in\I_n$, and $V=e^{-iX}$ then, the interlacing inequalities show that $s_j((1-V)|_\cS)\leq s_j(1-V)$, for $j\in\I_k$.
Since $\la_j(X)/2\in[-\pi/2,\pi/2]$ we see that $|\sin(\la_j(X)/2)|=\sin(|\la_j(X)|/2)$; then, straightforward computations show that $s_j(1-V)=2\,\sin(s_j(X)/2)$, for $j\in\I_n$.
Hence, from Davis and Kahan results it can be deduced that 
\beq\label{eq desi ver deb de DK}
\theta_j = s_j(Z) \leq s_j(X) \peso{for} j\in\I_k\,.\eeq 
The following result is related to  the inequalities in Eq. \eqref{eq desi ver deb de DK} 
in a similar sense as previous cases: better bounds (by Proposition \ref{pro com spr y s}) 
with respect to a weaker order (weak majorization instead of entrywise inequalities). 	\EOE
\end{rem}

\begin{teo} \label{angsprm}
Let $\ese,\ete \inc \C^n$ be subspaces 
and let $X\in \cH(n)$ be such that $e^{iX}\ese=\ete$. Then
\rm
\beq \label{eq conjp3}
 \Theta(\cS\coma \cT) \prec_w \frac12 \,\spr^+(X) 
\ .
\eeq
\end{teo}
\proof 
Let $S\in \cM_{n,k}(\C)$ 
be an isometry such that $R(S) = \ese$. 
We consider the smooth curve $T(\cdot):[0,1]\rightarrow M_{n,k}(\C)$ 
given by $T(t)= e^{i\,t\,X}S$, for $t\in [0,1]$; we also set 
$\cT(t)=R(T(t))\inc\C^n$, for $t\in [0,1]$. Notice that 
$T(0)=S$, $\cT(0)=\ese$ and  $\cT(1)=\cT$. Since each $T(t)$ is an isometry, 
the function $\Theta(\cdot):[0,1]\rightarrow [0,\pi/2]^k$ given by 
$\Theta(t)=\Theta(\ese\coma \ete(t)) = \arccos (s\ua (S^*T(t)\,)\,) $ for $t\in [0,1]$, is continuous and $\Theta(0)=0$. Then, as a consequence of the triangle inequality for principal angles \cite[Theorem 1]{QZLi} we see that 
\beq\label{eq desi trian ang}
\barr{rl}
\Theta(\ese\coma \ete)&\prec_w \suml_{j=0}^{m-1} \Theta(\ete(\frac j m)\coma \ete (\frac {j+1} m))
\earr
\eeq
Notice that $T(t+h)=e^{i\,t\,X}\,T(h)$ with $e^{i\,t\,X}\in\cU(n)$, for $t,\,h,\,t+h\in [0,1]$; thus, we see that $\Theta(\ete(\frac j m)\coma \ete (\frac {j+1} m))=\Theta(\ete(0)\coma \ete (\frac {1} m))$, for each $j\in\I_{n-1}$. Next, we show that 
\beq\label{eq ese etedeene}
\barr{rl}
\Theta(\ese\coma \ete(\frac 1m))&
\prec_w \frac {1} {2m} \ \sprm(X)+ O(m) \peso{with} \lim\limits_{m\rightarrow \infty} m\,O(m)=0\,.
\earr
\eeq 
Indeed, since $\frac{d}{dt}\sin(t)|_{t=0}=1$ and $\sin(0)=0$ we see that 
\beq\label{eq ang vs sen1}
\barr{rl}
\Theta(\cS,\, \ete(\frac{1}{m}))&
=\sin(\Theta(\cS,\, \ete(\frac{1}{m})))+O_1(m)
\peso{with}\lim\limits_{m\rightarrow \infty} m\,O_1(m)=0\,.
\earr
\eeq
Let $S_\perp \in \cM_{n,n-k}(\C)$ be an isometry such that $R(S_\perp) = \ese\orto$. 
For every $m \in \N$, since $T(\frac 1m )$ is an isometry, 
then Eq. \eqref{angles} assures that 
$$ 
\barr{rl}
\sin (\Theta(\cS\coma \cT(\frac 1 m)))&
=s(T(\frac 1 m) ^*S_\perp)=s(S\, e^{\frac {i}{m}\,X}\, S_\perp)\in ([0,1]^{k})\da\,.
\earr
$$
Since $S\, e^{i\,t\,X}\, S_\perp|_{t=0}=0$, and $\frac{d}{dt}(S\, e^{i\,t\,X}\, S_\perp)|_{t=0}=i\,S\,X\,S_\perp$ then, we have that 
\beq\label{eq ang vs sen2} 
S\, e^{\frac i m \,X}\, S_\perp =
\frac i m \,  S\, X\, S_\perp  + O_2(m) \peso{with} \lim_{m\rightarrow \infty}m\,O_2(m)=0\,.
 \eeq
Eqs. \eqref{con 2 gral}, \eqref{eq ang vs sen1} and \eqref{eq ang vs sen2} together with 
Weyl's inequality (Theorem \ref{teo ag}) 
imply that 
$$
\barr{rl}
\Theta(\cS,\, \ete(\frac{1}{m}))&
\prec_w \frac 1{2m} s\big(\, 2\, S\, X\, S_\perp \,\big)  + s\big(\,O_2(m) \,\big) + O_1(m) \\&\\ 
&\stackrel{\eqref{con 2 gral}}{\prec_w} \frac{1}{2m}\, \sprm(X) + O(m)\,,
\earr
$$
which shows that Eq. \eqref{eq ese etedeene} holds. 
Then, by Eq. \eqref{eq desi trian ang} we get that 
$$
\Theta(\ese\coma \ete)\prec_w \frac 1 2 \,\sprm(X)+ m\, O(m)\,.
$$ The result now follows by taking the limit when $m\rightarrow \infty$.
\qed

\pausa As a consequence of Theorem \ref{angsprm} we strengthen \cite[Theorem 6]{QZLi}.

\begin{cor}
Let $\ese,\ete \inc \C^n$ be $k$-dimensional subspaces 
and let $X\in \cH(n)$ be such that $e^{iX}\ese=\ete$. If $U=e^{i\,Z}\in\cU(n)$ is a direct rotation from $\ese$ onto $\ete$, for $Z\in\cH(n)$, then
$$ s(Z)\prec_w \frac 1 2 |\spr (X)|\prec_w s(X)\,.$$
\end{cor}

\proof If $U\in\cU(n)$ is a direct rotation from $\ese$ onto $\ete$ then, we have seen that 
$U=e^{i\,Z}$ for $Z\in\cH(n)$ such that 
$$\la(Z)=(\Theta(\ese\coma \ete)^*,-\,\Theta(\ese\coma \ete)^*,0)  \ ,
\peso{where} \Theta(\ese\coma \ete)^*=(\theta_i)_{i\in\I_r}\in \R^r
$$ 
denotes the vector with the positive principal angles, for some $r\leq [n/2]$. 
Hence, $s(Z)=(\Theta(\ese\coma \ete)^*,\Theta(\ese\coma \ete)^*,0)\da\in (\R^n)\da$.
On the other hand, by Theorem \ref{angsprm} we get that 
$$s(Z)=
(\Theta(\ese\coma \ete)^*,\Theta(\ese\coma \ete)^*,0)\prec_w 
\frac 1 2 (\sprm(X)\coma \sprm(X))\da =\frac 1 2 |\spr (X)| \da \,.
$$ By Proposition \ref{pro com spr y s} we get that $\frac 1 2 |\spr (X)|\prec_w s(X)$ and the result now follows from these two submajorization relations.
\qed

\begin{rem}\rm 
Let $A = \bm {cc} a_1&b\\  \bar{b} & a_2 \em\in \cH(2)$. Then 
\beq\label{spr 2x2}
\sprm(A) =  \la_1(A)-\la_2(A) = \big[ \,(\tr\, A)^2  -4\,\det \, A \big]\rai = 
\big(\, (a_1-a_2)^2 + 4|b|^2\,\big)\rai \in \R_{\ge 0}\ .
\eeq
More generally, if we consider
$A =\bm{cc}A_1&B \\B^*&A_2 \em \barr {c}\C^k\\ \C^k\earr \in \cH(2k)\,,$ it is 
natural to wonder whether an inequality of the form
\beq\label{conj tru}
\la\left(\big[\, (A_1-A_2)^2 +4 \, B^*B\big]\rai\right) \prec_w \sprm(A) \ 
\eeq holds true. Notice that this inequality would be an improvement of both 
Theorems \ref{vale con 2 gral} and \ref{lem prelim1}. 
Nevertheless, it turns out that Eq. \eqref{conj tru} does not hold in general.
For example, 
$$
A= \bm{cccc}
1 &    2 &   1 &    2 \\
2 &    1 &   1 &    0 \\
1 &    1 &   2 &    0 \\
2 &    0 &   0 &    2 \em \barr {c}\C^2\\ \ \\\C^2\earr   
\peso{has} \sprm (A)=(6.2714 \coma 1.6339) \ ,
$$
and $\la\left(\big[\, (A_1-A_2)^2 +4 \, B^*B\big]\rai\right)  =(4.7599 \coma 3.3680)$. 
But 
$$
\tr \,\sprm (A) = 7. 9053 < 8.1279  = \tr\left(\big[\, (A_1-A_2)^2 +4 \, B^*B\big]\rai\right) \ . 
$$
In what follows we state a weak version of Eq. \eqref{conj tru} that holds true in the general case.
\EOE
\end{rem}

\medskip

\begin{pro}\label{pro ver debil de desi caso cuadrado1}
Let $A =\bm{cc}A_1&B \\B^*&A_2 \em \barr {c}\C^k\\ \C^k\earr \in \cH(2k)$. Then \rm
\beq\label{eq weak version1}
\la \big(\,(A_1-A_2)^2 +4 \,\text{Re}(B)^2\,\big) \prec_w \spr^2(A) 
\ .
\eeq
\end{pro}

\proof
We conjugate by the unitary matrix $U = \bm{cc}0&I\\-I&0\em \in \cU(2k)$,  
$$
U^*AU = \bm{cc}A_2&-B^*\\-B&A_1 \em \peso{and get} 
D\igdef A-U^*AU = \bm{cc}A_1-A_2& 2\,\text{Re}(B)\\ 2\,\text{Re}(B) &A_2-A_1 \em \ . 
$$
By Weyl's inequality (item 1 in Theorem \ref{teo ah}),  
$$
\la(D) \prec \la(A)+\la(-A)=
(\spr(A)\coma -\spr\ua(A)\,)\in (\R_{\ge 0}^{2k})\da \ .
$$ 
Note that 
$$
D^2 = \bm {cc}(A_1-A_2)^2 +4 \,\text{Re}(B)^2 &* \\ * &(A_1-A_2)^2 +4 \,\text{Re}(B)^2 \em \ .
$$
If $E\igdef (A_1-A_2)^2 +4 \text{Re}(B)^2$, then using 
Remark \ref{convfunction} 
and Theorem \ref{teo ah} (see the Appendix),  
\beq\label{con E}
\left( \la(E)\coma \la(E)\,\right)\da  
\stackrel{\eqref{pinch}}{\prec} \la(D^2) \prec_w 
(\sprm(A)^2,\sprm(A)^2)\,,
\eeq
since $t \mapsto t^2 $ is a convex map. 
Clearly, Eq. \eqref{eq weak version1} follows from Eq. \eqref{con E}. 
\qed

\pausa We point out that Proposition \ref{pro ver debil de desi caso cuadrado1} does not imply neither 
Theorem \ref{vale con 2 gral} nor Theorem \ref{lem prelim1} (since the relation $\prec_w$ is not  preserved by taking square roots).

\def\Par{\big(\,}
\def\Pal{\big)\,}

\section {Reformulations of the key inequality}

In what follows we obtain a series of results that are consequences of the key inequality (Theorem \ref{vale con 2 gral}). Then we show that the key inequality is actually equivalent to several of these derived inequalities (see Theorem \ref{teo muchas equivalencias} below).

\subsection{Generalized commutators and unitary conjugates}\label{sec 41}

We begin this section with the following inequality for singular values of generalized commutators in terms of the spectral spread.

\begin{teo}\label{la conj}
Let $A_1\coma A_2 \in \cH(n)$ and $X\in \matn$. Then
\beq\label{conj}
s(A_1 \, X-XA_2)\prec_w s(X)\ \sprm(A_1\oplus A_2)\,.
\eeq
\end{teo}
\begin{proof}
Consider the matrices $C= A_1X-XA_2 \in \matn$, 
$$
B= A_1\oplus A_2 =\bm{cc}
A_1 & 0 \\ 0 & A_2
\em
\in \cH(2n) \py Z = \bm{cc}
0 & X \\ -X^* & 0
\em
\in i\,\cH(2n) \ .
$$
Note that 
$$
B\ Z -Z \ B =
\bm{cc}
0 & A_1X-XA_2 \\ X^*A_1- A_2X^*  & 0
\em
= \bm{cc}0&C\\C^*&0\em 
\in \cH(2n)
$$
Therefore, by Proposition \ref{hat trick como en el futbol} we get that 
$$
\la \Big(\,B\ Z -Z \ B\,\Big)
= \Big(\, s(C)\coma -s(C)\Big)\,\da\ .
$$
Fix $1\leq k\leq n$. Then, the previous identity shows that 
$$
\sum_{j=1}^ks_j(A_1X-XA_2)=
\sum_{j=1}^k \la\da_j\Big(\, B\ Z -Z \ B\Big)\ .
$$
Then, there exists a projection
$P\in \cH(2n)$ with $\rk\, P = k$ such that 
\beq\label{con tr}
\sum_{j=1}^ks_j(A_1X-XA_2)
=\tr\Big(\,\big[\,B\ Z -Z \ B\,\big]\ P\Big)\,\,.
\eeq
Using that 
$
\tr \Par \big[\,B\ Z -Z \ B\,\big]\ P\Pal= 
\tr \Par \big[\, P\ B -B\ P\,\big]\ Z \Pal \ ,
$
and that $$
|\tr(VW)|\leq \tr \Par s(VW)\Pal\leq \tr \Par s(V)\ s(W) \Pal
$$ for every $V\coma W \in \matrec{2n}$, from Eq. \eqref{con tr} we get 
$$
\sum_{j=1}^ks_j(A_1X-XA_2)\leq 
\tr \Par s(P\ B -B\ P\,) \ s(Z )\Pal \ .
$$
Now, notice that if we consider the block matrix representation (obtained by changing the orthogonal decomposition of $\C^{2n}$)
$$
P=\bm{cc} 
I & 0 \\ 0 & 0
\em
\barr{c} k \\ 2n-k \earr \py
B = A_1\oplus A_2=\bm{cc} 
S_{11} & S_{12} \\ S_{12}^* & S_{22}
\em 
$$
then, using these new representations we have that
$$
P\ B - B\ P
=\bm{cc} 
0 & S_{12} \\ -S_{12}^* & 0
\em 
\ .
$$
Hence 
$
s\Par PB-BP 
\Pal=
\Par s(S_{12}^*)\coma s(S_{12})\Pal\da 
$.  
Thus, by  Theorem \ref{vale con 2 gral},
\beq\label{eq aca se usa 35}
s(S_{12}^*)\prec_w \frac{1}{2}\ \sprm \, B \implies s\Par PB-BP \Pal\prec_w\frac{1}{2}\ 
\Par \sprm \, B \coma \sprm \, B\Pal 
\ .
\eeq
Using these facts, that $s( Z )=\Par s(X)\coma s(X)\Pal\da$ and item 5 in Lemma \ref{lem submaj props1},  we can now see that
\beq\label{acot}
\barr{rl}
s\Par PB-BP\Pal \  s(Z )& \prec_w \frac{1}{2}\ 
\Par \sprm \, B\coma \sprm \, B\Pal\da\  s(Z ) \\&\\
&= \frac{1}{2}\ 
\Par \sprm \, B\  s(X) \coma \sprm \, B\,\  s( X)\Pal \da 
\igdef \rho
\ .\earr
\eeq
Note that $\rk \, P=k\implies \rk \,(PB-BP)\le 2k \implies 
s_{2k+1}\big(\, PB-BP\,\big)  = 0$. 
Then
\begin{eqnarray*}
\sum_{j=1}^ks_j(A_1X-XA_2)&\leq& 
\tr\Par s\Par PB-BP\Pal 
\  s(Z )\Pal  
\\ &=& \sum_{j=1}^{2k} s_j\big(\,PB-BP\,\big) \, s_j(Z )  
\\&\stackrel{\eqref{acot}}{\le} &  \sum_{j=1}^{2k} \rho_j =
\sum_{j=1}^k s_j(X) \ \sprm_j (A_1\oplus A_2) \ , 
\end{eqnarray*}
for every $k \in \I_n\,$. This shows Eq. \eqref{conj}. 
\end{proof}

\begin{rem}\label{rem implicancias 1}
 By inspection of the proof of Theorem \ref{la conj}, we see that this result follows from the application of Theorem \ref{vale con 2 gral} in Eq. \eqref{eq aca se usa 35}. On the other hand, 
Theorem \ref{la conj} extends Theorem \ref{lem prelim1} (by taking $X=I$ in Eq. \eqref{conj}); in particular, the inequality in Eq. \eqref{conj} is sharp (see the end of Remark \ref{rem 213}). In the next section we will see that all these results are actually equivalent. \EOE
\end{rem}

\medskip

\begin{cor}\label{coro raro}
Let $A \coma X  \in \cH(n)$. Then 
\beq\label{sprm X}
s(A X-XA)\prec_w \sprm(X) \  \sprm \,(A\oplus A) 
\ .
\eeq
If $X \in \matn^+$ then also 
\beq\label{medio}
s(A X-XA)\prec_w 
\frac{1}{2}\,\|X\|\,  \sprm \,(A\oplus A)\ .
\eeq
\end{cor}
\proof
If we let $Y=X-\la\,I\in\cH(n)$ for some $\la \in \R$, then 
 $AY-Y A=AX-XA$. Hence, by Theorem \ref{la conj} we see that 
$$
s(AX-XA)\prec_w s(X-\la\,I)\  \sprm(A\oplus A)\,.
$$
Take $\la = \la _{k+1}(X)$, for $k = [\frac n2]$. By 
Eq. \eqref{s vs spr} we get that $s(X-\la)\prec_w \sprm(X)$, and Eq. \eqref{sprm X} follows from
these facts together with item 5 in Lemma \ref{lem submaj props1}.
If $X\in\matn^+$ we take $\la = \frac{\|X\|}{2}$, so we have that 
$s(Y) \leqp \|Y\| \uno = \frac{\|X\|}{2} \, \uno $ and hence $s(Y)\prec_w \frac{\|X\|}{2} \, \uno $. This fact together with item 5 in Lemma \ref{lem submaj props1} imply  Eq. \eqref{medio} above. \qed

\medskip

\begin{cor}\label{teo gen Kitt 1} Let $A\coma  B \in \matn^+$ and $X\in\matn$. Then 
\beq \label{con s}
 s(A X-XB)\prec_w s(X)\ s( A \oplus B)\,.
\eeq
In particular we get that, for every unitary invariant norm $N$,  
\beq \label{con nui AB}
N(AX-XB) \le  \|X\| \, N\,( A \oplus B) \ . 
\eeq
\end{cor}
\proof
Use Theorem \ref{la conj} and recall that if 
$ C\in \matrec{2n}^+$, then $\spr (C) \leqp \la(C) = s(C)$. \QED

\medskip

\pausa
Let $A,\,X\in\cH(n)$ and let $U:= e^{i\,X}\in\matu$. As mentioned in item 7 in Remark \ref{rem cosas elemen sobre spread} (or as a consequence of Eq. \eqref{Peron} in Theorem \ref{lem prelim1}) we get that
$$ 
s(A- UAU^*)\prec _w \spr^+(A\oplus UAU^*) =  \spr^+(A\oplus A)=|\spr(A)|\, .
$$ 
Nevertheless, in case $X$ is close to a (real) multiple of the identity then
$U=e^{i\,X}$ is close to a multiple of the identity as well; hence, in this case we 
would expect $A$ and $U^*AU$ to be close, too. Similarly, in case $A$ is close to a (real) multiple
of the identity, then we would expect $A$ and $U^*AU$ to be close. The following result provides  quantitative estimates that deal with these situations.

\begin{teo}\label{Teotomaestok} \rm
Let $A,\, X\in \H(n)$ and  $U=e^{i\,X}\in \cU(n)$. Then \rm 
$$
s(A-U^*AU)\prec_w s(X) \  \spr^+ (A\oplus A) \py
s(A-U^*AU)\prec_w s(A) \  \spr^+ (X\oplus X)\,.$$
\end{teo}
\proof We prove the first inequality; the proof of the second inequality is similar 
(it changes only after Eq. \eqref{AX-XA}\,) 
and the details are left to the reader. Let $U(\cdot):[0,1]\rightarrow \cH(n)$ be the smooth function given by 
$A(t)=e^{-i\,t\,X}\,A\,e^{i\,t\,X}$,  for $t\in [0,1]$. Notice that 
$A(0)=A$ and $A(1)=U^*A\,U$; using Weyl's inequality for singular 
values (item 1 in Theorem \ref{teo ah}) 
\beq \label{eq desi weyl en la curva conj unit}
\barr{rl}
s(A-U^*AU)\prec_w \suml_{j=0}^{m-1} s\big(\,A(\frac j m) - A(\frac {j+1} {m}) \,\big)
\peso{for every} m\in \N  \,.
\earr
\eeq 
Notice that $A(t+h)=e^{-i\,t\,X}\,A(h)\,e^{i\,t\,X}$ with $e^{i\,t\,X}\in\cU(n)$, for $t,\,h,\,t+h\in [0,1]$. Thus 
\beq\label{con m}
\barr {rl} s(A(\frac j m) - A(\frac {j+1} {m}))& 
=s(A - A(\frac 1 m) )\peso{for}j\in \I_{m-1}  \ . \earr
\eeq
Since $A-A(0)=0$ and $\frac{d}{dt} A(t)|_{t=0}=i\,(AX-XA)$ we get that 
\beq\label{AX-XA}
\barr{rl}
s(A - A(\frac 1 m) )&
=\frac 1 m \, s(AX-XA)+O(m)\peso{with}\lim\limits_{m\rightarrow \infty} m\,O(m)=0\,.
\earr
\eeq
Hence, by Theorem \ref{la conj}  we have that 
\beq\label{eq el teo 41 se aplica aca}
\barr{rl}
s(A - A(\frac 1 m) )&\prec_w \frac 1 m \, s(X)\, 
\  \, \sprm(A\oplus A)+ O(m)\,. 
\earr
\eeq 
Therefore, by Eq.'s \eqref{eq desi weyl en la curva conj unit} and \eqref{con m} 
we have that, for sufficiently large $m$,
$$
s(A-U^*AU)\prec_w s(X)\, \  \, \sprm(A\oplus A)+  m\,O(m)\,.
$$
The statement now follows by taking the limit $m\rightarrow \infty$ in the expression above.
\qed

\begin{rem}\label{rem implicancias 2}
By inspection of the proof of Theorem \ref{Teotomaestok}, we see that this result follows from the application of Theorem \ref{la conj} in Eq. \eqref{eq el teo 41 se aplica aca}. 
\end{rem}

\subsection {Equivalences of the inequalities}

In this section we show the equivalence of several of the main results obtained in Sections \ref{sec 31} and \ref{sec 41}.

 \begin{teo}\label{teo muchas equivalencias}
The following inequalities are equivalent:
\ben
\item Given $A_1,\, A_2\in\cH(n)$ then 
\beq\label{Peron vuelve}
s(A_1-A_2) \prec_w \sprm \,(A_1\oplus A_2)\ .
\eeq
\item Given $A =\bm{cc}A_1&B \\B^*&A_2 \em \barr {c}\C^n\\ \C^n\earr \in \cH(2n)$ then 
\beq\label{con 2 bis}
2\,s(B) \prec_w \sprm(A) \ .
\eeq
\item 
Given $A_1\coma A_2 \in \cH(n)$ and $X\in \matrec{n}$ then
\beq\label{conj bis}
s(A_1 \, X-XA_2)\prec_w s(X)\ \sprm(A_1\oplus A_2)\,.
\eeq
\item Given $A,\,X\in\matsa$ and $U=e^{i\,X}\in\cU(n)$ then
\beq\label{conj bis2}
s(A-U^*AU)\prec_w s(X) \  \spr^+ (A\oplus A) 
\,. \eeq
\item Given subspaces  $\ese,\ete \inc \C^n$ 
and $X\in \cH(n)$ such that $e^{iX}\ese=\ete$, then
\rm
\beq \label{eq conjp3equiv}
 \Theta(\cS\coma \cT) \prec_w \frac12 \,\spr^+(X) \ .
\eeq
\een
\end{teo}

\proof 
1 $\implies$ 2: 
Let $B =U |B| $ be the polar decomposition of $B$, with $U\in \cU(n)$. 
If 
$$
W= \bm{cc}U&0\\0&I\em \in \cU(2n) \peso{then}
W^*AW = \bm {cc}U^*A_1U& U^*B \\ B^*U &A_2\em  = \bm {cc}U^*A_1U& |B| \\ |B| &A_2\em \ .
$$
Since $\spr(A)=\spr(W^*AW)$ and $s(B)=s(|B|)$, in order to show Eq. 
\eqref{con 2 bis} we can assume that $B\in \cH(n)$. 
In this case, taking  $R=\bm{cc}0&I\\I&0\em\in  \cU(2n)\cap\cH(2n)$, we have that
$$
RAR 
= \bm {cc}A_2& B\\ B &A_1\em 
\implies
\frac{A+RAR}{2}=\bm {cc}\frac{A_1+A_2}2 &B\\ B &\frac{A_1+A_2}2\em  \ . 
$$
By Weyl inequality (item 1 in Theorem \ref{teo ah}),  
$$\la\left(\frac{A+RAR}{2}\right) \prec \frac{\la(A)+\la(RAR)}2=\la(A)
\stackrel{\eqref{mayo spr}}{\implies} 
\sprm(\frac{A+RAR}{2})\prec_w\sprm(A)\,.$$ Therefore, 
in order to show Eq. \eqref{con 2 bis} we can assume that $B\in \cH(n)$ and $A_1=A_2\,$.

\pausa
Take now $Z = \frac 1{\sqrt2}\bm{cc}I&I\\-I&I\em\in \cU(2n)$. Since now 
$A= \bm {cc}A_1&B\\ B &A_1\em$, then
\beq\label{ZAZ}
Z^*AZ = 
\bm{cc}A_1-B&0\\0&A_1+B\em 
\py \sprm(Z^*AZ)=\sprm(A)\ .
\eeq
Hence $2s(B) = s\big(\, [A_1-B]-[A_1+B]\,\big) 
\stackrel{\eqref{Peron vuelve}}{\prec_w} \sprm  (\, [A_1-B]\oplus [A_1+B]\,\big)  
\stackrel{\eqref{ZAZ}}{=} \sprm(A)$.

\pausa We have already shown (see Remarks \ref{rem implicancias 1} and \ref{rem implicancias 2}) 
that 2 $\implies$ 3 $\implies$ 4.

\pausa 4 $\implies$ 3: We first consider the case $A_1=A_2=A$ and $X\in\cH(n)$. In this case we consider $D(t)=A-e^{-i\,t\,X}\,A\,e^{i\,t\,X}$, for $t\in [0,1]$.
Then, $D(\cdot)$ is a smooth function such that $D(0)=0$ and $D'(0)=-i\,(AX-XA)$.
Hence, by Weyl's inequality, we have that 
$$
s(t\,(AX-XA))= s(A-e^{-i\,t\,X}\,A\,e^{i\,t\,X})+O(t)\peso{with}\lim_{t\rightarrow 0^+}\frac{O(t)}{t}=0\,.
$$
Using Eq. \eqref{conj bis2} we now see that 
$$s(AX-XA)\prec_w s(X)\  \sprm(A\oplus A)+\frac{O(t)}{t}\,.$$
Then 3. follows by taking the limit $t\rightarrow 0^+$, when $A_1=A_2$ and $X\in\cH(n)$. 
 For the general case, consider
$$A =\bm{cc}A_1 &0\\0&A_2\em\in \cH(2n) \py \hat X=
\bm{cc}0&X\\X^*&0\em \in \cH(2n)\,.$$
Notice that 
$$
A\,\hat X-\hat X\,A=\bm{cc}0&A_1X-XA_2\\-(A_1X-XA_2)^*&0\em 
$$
Hence, by Proposition \ref{hat trick como en el futbol} and the previous facts,
$$
s(A\,\hat X-\hat X\,A)=(s(A_1X-XA_2)\coma s(A_1X-XA_2))\da\prec_w s(\hat X)\  \sprm(A\oplus A)\,.
$$ Notice that Eq. \eqref{conj bis} follows from the previous submajorization relation, since
$$
s(\hat X)=(s(X)\coma s(X))\da \py 
\sprm(A\oplus A)=(\sprm(A_1\oplus A_2)\coma \sprm(A_1\oplus A_2))\da \ .
$$ 
Since 3 $\implies$ 1 (by taking $X=I$) we see that 1 -- 4 are equivalent. 
We have shown (in the proof of Theorem \ref{angsprm}) that 2 $\implies$ 5. Thus, we are left to show that 5 $\implies$ 2. 
Indeed, fix $A\in\cH(2n)$ as in 2. Let $\cS=\C^n\oplus 0\subset \C^n\oplus \C^n$ and let $\cT(t)=e^{i\,t\,A}\cS$, for $t\in [0,1]$. Let 
$$
S=\bm{cc}I_n\\ 0\em\ , \ S_\perp=\bm{cc}0\\ I_n\em\in M_{2n,n}(\C) \py T(t)=e^{i\,t\,A}S 
\peso{for} t\in [0,1]\,.$$
Hence, by Eq. \eqref{angles}, in this case we have that
 $$\sin(\Theta(\cS\coma \cT(t)))=s(T(t)^*\,S_\perp)=s(S\,e^{i\,t\,A}\,S_\perp)\,.$$
Since $T^*(0)\,S_\perp=0$ and $\frac{d}{dt}(T^*(t)\,S_\perp)|_{t=0}= i\,S\,A\,S_\perp=i\,B\in \cM_n(\C)$
we get that 
$$
t\,s(B)= \sin(\Theta(\cS\coma \cT(t)))+O_1(t)\peso{with}\lim_{t\rightarrow 0^+} \frac{O_1(t)}{t}=0\,.$$
On the other hand, since $\Theta(\cS\coma \cT(t))$ is a continuous function of $t\in[0,1]$ such that $\Theta(\cS\coma \cT(0))=0$, then we have that 
$$
\sin(\Theta(\cS\coma \cT(t)))=\Theta(\cS\coma \cT(t))+O_2(t)\peso{with} \lim_{t\rightarrow 0^+} 
\frac{O_2(t)}{t}=0\,.$$
The previous facts together with \eqref{eq conjp3equiv} in 5. show that 
$$
t\,s(B)\prec_w \frac t 2 \, \sprm(A)+O_1(t)+O_2(t)\implies 2\, s(B)\prec_w \sprm(A)+\frac{2}{t} \,(O_1(t)+O_2(t))\,.
$$ Then 2. follows from the previous inequality, by taking the limit $t\rightarrow 0^+$.
\qed

\begin{rem}
We point out that item 2 of Theorem \ref{teo muchas equivalencias} 
implies Theorem \ref{vale con 2 gral} in its general form (where the block $B$ can be rectangular). This follows from Eq. \eqref{comprimido} in Remark \ref{rem cosas elemen sobre spread} (the details are left to the reader). \EOE
\end{rem}

\subsection{Concluding remarks}

We have developed several aspects of the spectral spread of Hermitian matrices introduced in \cite{AKFEM}, which is a natural vector valued measure of the dispersion of the spectra. We have also connected our work with well established research topics in matrix analysis. We have obtained several inequalities involving the spectral spread; in particular, we have obtained sharp inequalities for  generalized commutators of the form $A_1X-XA_2$, for Hermitian matrices $A_1,\,A_2\in \cH(n)$ and arbitrary
$X\in \matn$. Generalized commutators appear, in a natural way, as derivative vectors of matrix-valued smooth curves. We expect that our results will have applications in matrix perturbation bounds for Hermitian matrices, obtained from a (differential) geometrical perspective. Indeed, in \cite{MSZ2} we have already applied the results herein and obtained some inequalities related to the bounds in Eq. \eqref{eq intro conjK} (see \cite{AKFEM}) using a geometrical approach. We point out that matrix sensitivity problems related with the variation of the eigenvalues and eigenspaces of (perturbations of) Hermitian matrices have applications in Machine Learning (tracking changes in data, see \cite{Bis}) while the bounds in the variation of angles between subspaces are of interest in Quantum Computing theory (see \cite{NiCh}). 

\medskip

\pausa {\bf Acknowledgments}. We would like to thank Professor X. Zhan and the anonymous reviewers for several useful suggestions that improved the exposition of the results in this work.

\section{Appendix}\label{sec appendixity}
\pausa
Here we collect several well known results about majorization, used throughout our work.
For detailed proofs of these results and general references in majorization theory see \cite{bhatia,HJ}. We begin with the Weyl's inequalities for singular values: 

\begin{teo}\label{teo ag}\rm Let $C,\,D\in \matn$. Then,
\ben
\item Weyl's additive inequality: $s(C+D)\prec_w s(C)+s(D)$;
\item Weyl's multiplicative inequality: $s(CD)\prec_w s(C)\, s(D)$.
 \qed\een
\end{teo}

\begin{teo}\label{teo ah} \rm Let $C,\, D\in \matsa$. Then,
\ben
\item Weyl's additive inequalities (for eigenvalues): 
\ben
\item 
$\lambda(C+ D)\prec\lambda(C)+\lambda(D)$;
\item  if $C\le D$ then $\la(C) \leqp \la(D)$; 
\een
\item 
Lidskii's additive inequality: $\lambda(C)-\lambda(D)\prec \lambda(C-D)\prec\lambda(C)-\lambda^{\uparrow}(D)$;
\item $|\la(C)-\la(D)|\prec_w s(C-D)$;
\item Let $\cP=\{P_j\}_{j=1}^r$ be a system of projections (i.e. they are mutually orthogonal projections on $\C^d$ 
such that $\sum_{i=1}^r P_i=I$). If 
\beq\label {pinch}
\cC_{\cP}(D)\igdef \sum_{i=1}^r P_i\,D\, P_i  \implies \lambda(\cC_{\cP}(D))\prec \lambda(D) 
\ . 
\eeq
\QED
\een 
\end{teo}
\pausa
In the next result we describe several elementary but useful properties of (sub)majorization between real vectors.

\begin{lem}\label{lem submaj props1}\rm 
Let $x\coma y\coma  z\coma w \in \R^k$. Then,
\ben
\item $x\da + y\ua\prec x+y\prec x\da+y\da$;
\item If $x\prec_w y$ and $y,\, z\in(\R^k)\da$ then $x+z\prec_w y+z$;
\item[3.] Moreover, if $z\coma w\in (\R^k)\da$,  $x\prec z$ and $y\prec w$ then $ x+ y\prec z+w$. 
\een
If we assume further that $x\coma y\coma  z\in \R_{\geq 0}^k$ then,
\ben
\item[4.] $x\da\, y\ua\prec_w x\, y\prec_w x\da\, y\da$;
\item[5.] If $x\prec_w y$ and $y,\, z\in (\R_{\geq 0}^k)\da$ then $x\, z\prec_w y\, z$.
\qed

\een
\end{lem}

\begin{rem}\label{remxleqy}
Let $x,y\in \R^k$. If $x\leqp y$ then, $$x^{\downarrow}\leqp y^{\downarrow}\  \text{ and } \ x\prec_w y \, .$$
\end{rem} 

\pausa
Recall that given $f: I \rightarrow \R$, where $I\subset \R$ is an interval, and $z=(z_i)_{i\in\I_k}\in I^k$ we denote $f(z)=(f(z_i))_{i\in\I_k}\in\R^k$.
\begin{rem}\label{convfunction}
Let $I\subset \R$ be an interval and let  $f: I \rightarrow \R$ be a convex function. Then, 
\ben
\item if $x,\, y\in I^k$ satisfy $x\prec y$ then $f(x)\prec_w f(y)$. 
\item If   $x\prec_w y$ but $f$ is further non-decreasing in $I$, then $f(x)\prec_w f(y)$. 
\EOE\een
\end{rem}

\pausa
Recall that a norm $N$ in $\matn$ is unitarily invariant (briefly u.i.n.) if 
$\nui{UAV}=\nui{A}$, for every $A\in\matn$ and $U,\, V\in\mathcal{U}(n)$.
Well known examples of u.i.n. are the spectral norm $\|\cdot\|_{sp}$ and the $p$-norms $\|\cdot\|_p\,$, for $p\geq 1$.

\begin{rem}\label{Domkyfan}\rm
It is well known that (sub)majorization relations between singular values of matrices are intimately related 
with inequalities with respect to u.i.n's. Indeed, given $A,\, B\in\matn$ the following statements are equivalent:
\ben 
\item For every u.i.n. $N$ in $\matn$
we have that $N(A)\leq N(B)$.
\item $s(A)\prec_w s(B).$ \EOE
\een 
\end{rem}

\begin{pro}\label{hat trick como en el futbol}\rm
Let $1\le k<n$,  $E\in \cM_{k\coma n-k}(\C)$ and 
$\hat{E}=\begin{pmatrix} 0&E\\ E^*&0
\end{pmatrix}\in \H(n)$. Then 
$$
\la(\hat E)= \big(\, s(E)\coma -s(E^*)\,\big)\da=\big(\,s(E)\coma -s\ua(E)\, \big)
\in(\R_{\geq 0}^n)\da\ ,
$$
with some zeros at the middle, in the rectangular case. \QED
\end{pro}

\end{document}